\documentclass[11pt]{amsart}

\usepackage[margin=1.5in,twoside=false]{geometry}\headsep = 0.5 in

\usepackage{amsmath,amssymb,amsthm,latexsym,graphicx,verbatim}
\usepackage[mathrefs,edtable]{lineno}
\usepackage{setspace}
\usepackage{multirow}
\usepackage{subfig,keyval,cancel,caption}
\usepackage[normalem]{ulem}
\usepackage[dvipsnames,usenames]{xcolor}

\usepackage[final,colorlinks=true]{hyperref}

\newtheorem{theorem}{Theorem}
\newtheorem{corollary}[theorem]{Corollary}

\newtheorem{proposition}[theorem]{Proposition}

\theoremstyle{remark}

\newcommand{\NN}{{\mathbb N}}
\newcommand{\ZZ}{{\mathbb Z}}

\newcommand{\bc}{\begin{center}}
\newcommand{\ec}{\end{center}}
\newcommand{\be}{\begin{enumerate}}
\newcommand{\ee}{\end{enumerate}}
\newcommand{\bi}{\begin{itemize}}
\newcommand{\ei}{\end{itemize}}

\newcommand{\mc}[1]{\ensuremath{\mathcal{#1}}}
\newcommand{\Po}{\ensuremath{\mathcal{P}}}

\newcommand{\mon}[1]{\text{Mon}(#1)}

\theoremstyle{comment}
\newtheorem*{dcomment}{\color{BrickRed}{Daniel Says}}
\newtheorem*{gcomment}{\color{BrickRed}{Gordon Says}}
\newtheorem*{mcomment}{\color{BrickRed}{Michael Says}}

\begin{document}
\title{Minimal covers of the prisms and antiprisms}
\author{Michael I. Hartley}
\address[Michael I. Hartley]{DownUnder GeoSolutions\\
80 Churchill Ave\\
Subiaco, 6008\\
Western Australia\\
E-mail: mikeh@dugeo.com}
\author{Daniel Pellicer$^{\dagger}$}
\address[Daniel Pellicer]{Instituto De Matematicas, Unidad Moreli\\
Antigua Carretera a Patzcuaro 8701\\
Col. Ex Hacienda de San Jose de la Huerta, 58089\\
Morelia, Michoacan, Mexico \\
E-mail: pellicer@matmor.unam.mx}
\author{Gordon Williams}
\address[Gordon Williams]{Department of Mathematics and Statistics\\
University of Alaska Fairbanks\\
PO Box 756660\\
Fairbanks AK 99775-6660\\
E-mail: giwilliams@alaska.edu\\
1+(907)347-4020\\
Corresponding Author}
\maketitle
\begin{abstract}This paper contains a classication of the regular minimal abstract polytopes that act as covers for the convex polyhedral prisms and antiprisms. It includes a detailed discussion of their topological structure, and completes the enumeration of such covers for convex uniform polyhedra. Additionally, this paper addresses related structural questions in the theory of string C-groups.\end{abstract}

\section{Introduction}
Symmetric maps on surfaces have been extensively studied, especially in the context of compact surfaces (see \cite{Bra27}, \cite{CoxMos80}). The geometric and combinatorial structure of the prisms and antiprisms have been studied since antiquity, and in modern times maps whose vertex figures are polygons have begun to be studied as abstract polyhedra (rank 3 polytopes)  ~\cite[\S 6B]{McMSch02}. While much work has been done on the study of regular abstract polytopes (the primary reference on the topic is \cite{McMSch02}), and there is increasingly large body of literature on the structure of chiral polytopes (abstract polytopes whose flags fall into two symmetry classes, with adjacent flags in different orbits, c.f. ~\cite{SchWei91},~\cite{Pel}),
 the study of less symmetric abstract polytopes is still in its early development.

A seminal paper in the study of less symmetric polytopes was Hartley's \cite{Har99a} discovery of how an abstract polytope may be represented as a quotient of some regular abstract polytope. This paper is part of an ongoing effort to better understand the nature of these quotient representations both geometrically as covering maps and algebraically via the group actions induced by the automorphism groups of the regular covers. In what follows we shall provide explicit descriptions of minimal regular covers of the $n$-prisms and $n$-antiprisms. 

Together with the results in \cite{HarWil10}, this provides a complete description of the minimal regular covers of the convex uniform polyhedra, where a uniform polyhedron has regular facets with an automorphism group that acts transitively on its vertices. We also provide the first description in the context of abstract polytopes of minimal regular covers for an infinite class of non-regular polytopes. 

\section{Background}
In the current work our focus is on the study of {\em abstract polyhedra}, which are defined by  restricting  the definition of {\em abstract polytopes} to the rank 3 case. Readers interested in the more general theory and definitions should see \cite[\S 2A]{McMSch02}, which we follow closely here. An abstract polyhedron \mc P is a partially ordered set, with partial order on the elements denoted by $\le$ satisfying the constraints P1-P4.
\begin{description}
\item[P1] It contains a unique minimum face $F_{-1}$ and a unique maximum face $F_{3}$.\item[P2] All maximal totally ordered subsets of \mc P, called the {\em flags} of \mc P, include $F_{-1}$ and $F_{3}$ and contain precisely $5$ elements.
\end{description}
As a consequence of P1 and P2, the ordering $\le$ induces a strictly increasing rank function on \mc P with the ranks of $F_{-1}$ and $F_{3}$ being $-1$ and $3$, respectively. Following the terminology of the inspiring geometric objects, the elements of ranks $0$, $1$ and $2$ of an abstract polyhedron are respectively called {\em vertices}, {\em edges} and {\em faces}.
Let $F,G$ be two elements of \mc P. We say $F$ and $G$  are {\em incident} if $F\le G$ or $G\le F$.
\begin{description}
\item[P3] The polyhedron \mc P is strongly connected (defined below).
\item[P4] The polyhedron \mc P satisfies the ``diamond condition'', that is, all edges are incident to precisely two vertices and two faces, and if a vertex $V$ is incident to a face $F$, then there exist precisely two edges incident to $V$ and $F$.
\end{description}
A {\em section} determined by $F$ and $G$ is a set of the form $G/F:=\{H\mid F\le H\le G\}$. A poset \mc P of rank $n$ is said to be {\em connected} if either $n\le1$ or $n\ge 2$ and for any two proper faces $F,G\in\mc P$ there exists a finite sequence of proper faces $F=H_{0},H_{1},\ldots,H_{k}=G$ of \mc P where $H_{i-1}$ and $H_{i}$ are incident for each $1\le i\le k$. A poset \mc P is said to be {\em strongly connected} if every section of \mc P is connected.

Let \mc P be an abstract polyhedron. 
The {\em vertex-figure} of \mc P at a vertex $v$ is the section $\mc P/v=\{F\in\mc P\mid v\le F\}$. The {\em degree} of a vertex $v$ is the number of edges containing $v$, and the {\em degree} of a face $f$ (sometimes called {\em co-degree} of a face) is the number of edges contained in $f$. 
Polyhedra for which the degree of every vertex is $p$ and the co-degree of every face is $q$ are said to be {\em equivelar} of {\em Schl\"afli type} $\{p,q\}$.

It follows from P4 that, given $i\in\{0,1,2\}$ and  a flag $\Psi$ of \mc P, there exists a unique flag $\Psi^{i}$ which differs from $\Psi$ only it its element at rank $i$. The flag $\Psi^{i}$ is called the {\em $i$-adjacent} flag of $\Psi$. The strong connectivity implies now that each face or vertex-figure of \mc P is isomorphic to a polygon as a poset.

A {\em rap-map} is a map between polyhedra preserving rank and adjacency, that is, each element is sent to an element of the same rank, and adjacent flags are sent to adjacent flags. An {\em automorphism } of a polyhedron \mc P  is a bijective rap-map of \mc P to itself. We denote the group of automorphisms of a polyhedron by $\Gamma(\Po)$ --- or simply by $\Gamma$ whenever there is no possibility of confusion --- and say that \mc P is {\em regular} if $\Gamma$ acts transitively on the set of flags of \mc P, denoted $\mc F(\Po)$. Familiar examples of geometric polyhedra whose face lattices are regular abstract polyhedra are the platonic solids and the regular tilings of the plane by triangles, squares or hexagons.

Throughout this paper we will use ``polyhedra'' to mean either the geometric objects or abstract polyhedra, as appropriate.

   A {\em string C-group} $G$ of rank $3$ is a group with distinguished
involutory generators $\rho_0, \rho_1, \rho_2$, where
$(\rho_0 \rho_2)^2 = id$, the identity in $G$, and $\langle \rho_0, \rho_1 \rangle \cap \langle \rho_1, \rho_2 \rangle
= \langle \rho_1 \rangle$  (this is called the {\em intersection condition}).

The automorphism group of an abstract regular polyhedron \mc P is always a string
C-group of rank $3$. In fact, given an arbitrarily chosen {\em base flag} $\Phi$ of \mc P, $\rho_{i}$ is taken to be the (unique) automorphism mapping $\Phi$ to the $i$-adjacent flag $\Phi^{i}$.
Furthermore, any string C-group of rank $3$ is the automorphism
group of an abstract regular
polyhedron \cite[Section 2E]{McMSch02}, so, up to isomorphism, there is a one-to-one correspondence between the string C-goups of rank $3$
and the abstract regular polyhedra.  Thus, in the study of regular abstract polyhedra we may either  work with the polyhedron as a poset, or with its automorphism group. 
We now review some of the relevant results and definitions from ~\cite{PelWil11,MonPelWil11a}.

The {\em monodromy group} $\mon{\mc P}:=\langle r_0, r_1, r_2 \rangle$ of a polyhedron $\mc P$ is the group of permutations on $\mc F(\mc P)$ generated by the maps $r_{i}:\Psi\mapsto\Psi^{i}$  (see \cite{HubOrbWei09}). It is important to note that these are {\em not} automorphisms of \mc P since they are not  adjacency preserving (compare the action of $r_{2}$ on $\Psi$ and $\Psi^{1}$). A string C-group $\Gamma=\langle \rho_{0},\rho_{1},\rho_{2}\rangle$ has a {\em flag action} on \mc P if there is a group homomorphism from $\Gamma\to\mon{\mc P}$ defined by $\rho_{i}\mapsto r_{i}$. Note also that the action of $r_{i}$ (and thus of the flag action) commutes with the automorphisms of any given polyhedron, so $(\Psi r_{i})\alpha=(\Psi \alpha)r_{i}$ and more generally, for all $w\in\mon{\mc P},\alpha\in\Gamma(\mc P)$ then $(\Psi w)\alpha=(\Psi\alpha)w$. Observe that in the case that there exists a group homomorphism from $\Gamma$ to $\mon(\mc P)$, it is contravariant since $(\Psi \rho_i) \rho_j = (\Psi^i) \rho_j = (\Psi \rho_j)^i = \Psi r_j r_i$. 

We say that the regular polyhedron $\mc P$ {\em covers}  $\mc Q$, denoted by $\mc P\searrow \mc Q$, if
$\mc Q$ admits a flag action from $\Gamma(\mc P)$. (This implies the notion of covering described in \cite[p. 43]{McMSch02}.) For example, if $p$ is the least common multiple of the co-degrees
of the faces of a polyhedron $\mc P$, and $q$ is the least common multiple of the vertex degree  of $\mc P$, then $\mc P$ is covered by the
tessellation \mc T of type $\{p, q\}$ whose automorphism group is isomorphic to the string Coxeter group
\[[p, q] := \langle \rho_0, \rho_1, \rho_2 \,|\, (\rho_0 \rho_2)^2 =
(\rho_0 \rho_1)^p = (\rho_1 \rho_2)^q = id \rangle.\]
Here the polyhedron \mc T can be viewed as a regular tessellation of the sphere, Euclidean plane or hyperbolic plane, depending on whether $\frac{1}{p}+\frac{1}{q}$ is bigger than, equal to, or less than $\frac{1}{2}$, respectively. We say that \mc P is a {\em minimal} regular cover of \mc Q if $\mc P\searrow\mc Q$ and if \mc R is any other regular polyhedron which covers \mc Q and is covered by \mc P, then $\mc P=\mc R$.

Let $\mc P\searrow\mc Q$, then by the main result of \cite{Har99} the structure of \mc Q is totally determined by the stabilizer $N$ of a specified base flag $\Phi\in\mc F(\mc Q)$ under the flag action of $\Gamma(\mc P)$. In fact, the elements of \mc Q are understood to be precisely the orbits of the elements of \mc P under the action of $N$.

Central to the identification and construction of minimal regular covers of polyhedra is the following theorem combining results from \cite{PelWil11,MonPelWil11a}.
\begin{theorem} Let \mc Q be an abstract polyhedron and \mon{\mc Q} its monodromy group. Then \mon{\mc Q} is a string C-group. Moreover, the regular abstract polytope \mc P associated with $\mon{\mc Q}$ is the minimal regular cover for \mc Q.
\label{th:monIsString}\end{theorem}
The proof of these facts in \cite{PelWil11} depends on the observation that
$\mon{\mc Q}\cong \Gamma/\text{Core}(\Gamma,N)$, where $\Gamma$ is the automorphism group of any regular cover of \mc Q, $N$ is the stabilizer in $\Gamma$ of a flag in \mc Q under the flag action of $\Gamma$ and the {\em core} is the largest normal subgroup of $\Gamma$ in $N$, denoted $\text{Core}(\Gamma,N)$.

It verges on folklore that whenever \mc P is regular, $\Gamma(\mc P) \cong \mon(\mc P)$ \cite{MonPelWil11a}.  
This leads to a useful reinterpretation of the condition for a regular polyhedron \mc P to be a cover of \mc Q. The fact that  \mc Q admits a flag action by $\Gamma(\mc P)$ is equivalent to observing that there is an epimorphism from $\mon {\mc P}$ to $\mon {\mc Q}$. Thus,
we find it more natural to understand the cover $\mc P\searrow \mc Q$ as an epimorphism of monodromy groups, instead of as a contravariant homomorphism from an automorphism group to a monodromy group. This perspective is motivated by the natural way in which $i$-adjacent flags of $\mc P$ are mapped into $i$-adjacent flags of $\mc Q$. Henceforth we shall proceed according to this notion and use the generators $r_0, r_1, r_2$ of $\mon{\mc P}$ instead of those of $\Gamma({\mc P})$ to denote the action on the flags of $\mc Q$.  For compactness of notation, we will frequently write $a,b$ or $c$ instead of $r_{0},r_{1}$ or $r_{2}$, respectively.

\section{On the Sufficiency of Generating Sets for Flag Stabilizers}
For a given abstract polyhedron $\mc P$, we define its flag graph $\mc{GF}(\mc P)$ as the
edge-labeled graph whose vertex set consists of all flags of $\Po$, where two vertices (flags) are joined by an edge labeled $i$
 if and only if they are $i$-adjacent for some $i=0, 1, 2$.

We recall a standard result from graph theory (see, e.g., \cite{Big71}):
\begin{theorem}Let $G$ and $G^{*}$ be dual planar graphs, and $T$ a spanning tree of $G$. Then the complement of the edges of $T$ is a spanning tree for $G^{*}$.
\end{theorem}

We also recall the following useful theorems from \cite{PelWil10}:
\begin{theorem} Let $T$ be a spanning tree in the flag graph $\mc{GF}(\mc Q)$ of $\mc Q$ rooted at $\Phi$, a specified (base) flag of \mc Q. For each edge $e=(\Psi,\Upsilon)$ of $\mc{GF}(\mc Q)$, define the walk $\beta_{e}$ as the unique path from $\Phi$ to $\Psi$ in $T$, across $e$ and followed by the unique path from $\Upsilon$ to $\Phi$. Let $w_{\beta_{e}}$ be the word in $\Gamma$ inducing the walk $\beta_{e}$.
Then $S=\{w_{\beta_{e}}:e\in\mc{GF}(\mc Q)\setminus T\}$ is a generating set for $Stab_{\Gamma}(\Phi)$.\label{t:treeGenerates}\end{theorem} Note that this is essentially just a restatement of the Reidemeister-Schreier algorithm for finding a generating set for the stabilizer of a vertex in the automorphism group of {\em any} finite graph (c.f. ~\cite{Con92a}), however this theorem extends the result to countably infinite graphs in the natural way. While the generating sets for the stabilizer of a base flag obtained in Theorem \ref{t:treeGenerates} are handy, a more natural way to construct elements of the stabilizer of a base flag in a tiling come from {\em lollipop} walks, that is, walks from the base flag to the cell of the flag graph corresponding to a vertex, edge or face of the tiling (the {\em stem} of the walk), around that cell, and back along the same path. The following theorem allows us to construct a generating set for the stabilizer of a base flag in a tiling using this more natural construction.
\begin{theorem}\label{l:FVOneEnough} Let \mc Q be a finite polyhedron with planar flag graph, $\Phi$ be a base flag for \mc Q, and let $\mc P \searrow \mc Q$ with  $\Gamma:=\text{Aut}(\mc P)$. Then $\text{Stab}_{\Gamma}(\Phi)$ admits a generating set containing no more than one generator for each vertex and face of \mc Q.\begin{proof} Let $T$ be a spanning tree in $\mc G\mc F (\mc Q)$, and $T^{*}$ the spanning tree for $\mc G\mc F (\mc Q)^{*}$ corresponding to the omitted edges of $T$ in $\mc F\mc G(\mc Q)$. Let $\{g_{x}\}$ be a set of elements of $\Gamma$ corresponding to the vertices, edges and faces  of \mc Q such that each $g_{x}$ is of the form
\begin{equation}\label{eq:wiandgi}
g_{x}=w_{x}(r_{i}r_{j})^{q_{x}}w_{x}^{-1}
\end{equation}
with $q_{x}$ is the degree of the node $x$ in $\mc G\mc F(\mc Q)^{*}$ corresponding to a vertex, edge or face of $\mc Q$, and $w_{x}$ induces a walk on the flag graph from $\Phi$ to a flag on the corresponding vertex, edge or face entirely contained in $T$. Then each $g_{x}$ corresponds to a lollipop walk in $\mc G\mc F(\mc Q)$. Let $G=\{g_{x}|x\in\mc Q\}\cup\{g_{x}^{-1}|x\in\mc Q\}$.

To prove the result, it suffices to show that each of the $w_{\beta_{e}}$ of Theorem \ref{t:treeGenerates} induced by $T$ may be obtained as a suitably ordered product of elements from $G$.  Let $e$ be an omitted edge of $T$ in $\mc F\mc G(\mc Q)$. We say that a node $x$ of $T^{*}$ is {\em enclosed} by $\beta_{e}$ if it is contained in the region bounded by $\beta_{e}$. An edge of $T^{*}$ is {\em enclosed} by $\beta_{e}$ if either of its endpoints is. Our proof proceeds by induction on the number of edges enclosed by $\beta_{e}$.

Suppose $\beta_{e}$ encloses a single edge of $T^{*}$, then $\beta_{e}$ encloses a single node $u\in T^{*}$ corresponding to a cell $U$ of $\mc F\mc G(\mc Q)$. To see this it suffices to observe that $\beta_{e}$ crosses $e^{*}$, and bounds a simply connected region and so cannot contain both endpoints of $e^{*}$. Thus $w_{\beta_{e}}$ induces a walk in $T$ to a flag on $U$, around the cell and back again. Thus $w_{\beta_{e}}=g_{u}$ (or its inverse) for some choice of $u$.

As an inductive hypothesis, suppose that if a word $\beta_{e}$ encloses at most $k-1 \ge 1$ edges of the dual $T^{*}$ of the tree $T$ in $\mc F\mc G (\mc Q)$, then it can be expressed as a product elements of $G$ with respect to the tree $T$.

\begin{figure}[htbp]
\begin{center}
\includegraphics[width=2in]{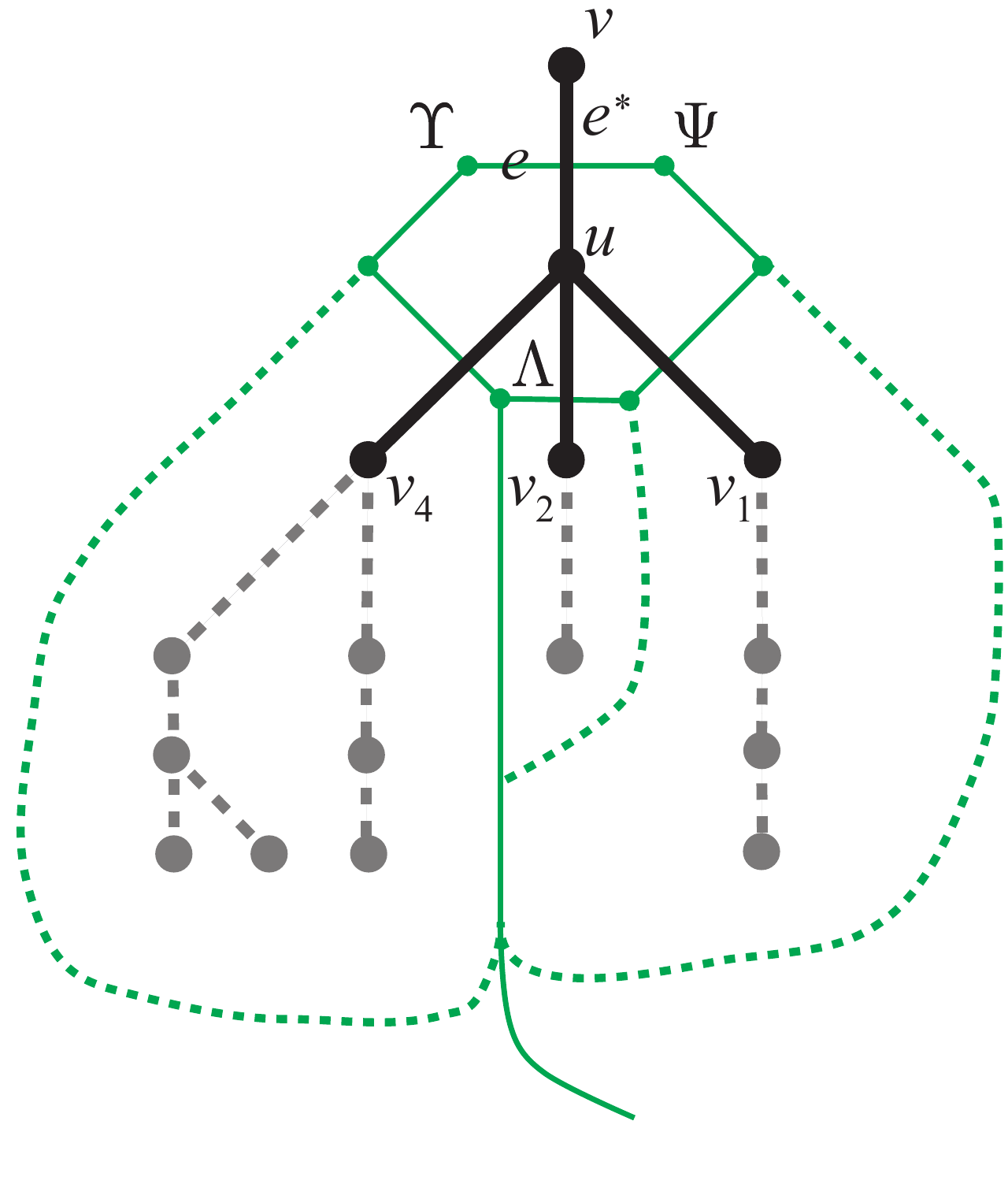}
\caption{A diagram of the inductive step in Theorem \ref{l:FVOneEnough} for an edge $e$ with node $u$ of degree $d=4$. Black and grey edges are in $T^{*}$, green edges are in $\mc G\mc F(\mc Q)$, solid green edges are traversed by $g_{u}$ and dashed grey edges correspond to portions of $T^{*}$ that may vary in size depending on the choice of $e$. Note that the ordering on the $v_{i}$ here corresponds to $w_{\beta_{e}}$ traversing the edge $e$ counterclockwise.}
\label{f:InductiveStep}
\end{center}
\end{figure}
Suppose $\beta_{e}$ encloses $k$ edges. Let $e^{*}$ be the edge in $T^{*}$ dual to $e$,  denote the endpoints of $e$ by $\Psi$ and $\Upsilon$ and without loss of generality suppose that $w_{\beta_{e}}$ traverses $e$ from $\Psi$ to $\Upsilon$ and let $u$ and $v$ be the endpoints of $e^{*}$, where $u$ is enclosed by $\beta_{e}$ and $v$ is not.  Let $U$ denote the cell of $\mc G\mc F(\mc Q)$ corresponding to $u$, and let $\Lambda$ denote the first vertex of $U$ traversed by $g_{u}$. The direction in which $w_{\beta_{e}}$ traverses the edge $e$ induces an orientation on the edges of $U$. Denote the degree of $u\in T^{*}$ by $d$. We index the edges incident to $u$ (other than $e^{*}$) and $e_{1}, e_{2},\ldots,e_{k-1},e_{k+1},e_{k+2},\ldots,e_{d}$ starting at $\Psi$ and running opposite the orientation induced on the edges of the cell $U$ in the order crossed so that $\Lambda$ is shared by $e_{k-1}$ and $e_{k+1}$. Label the dual edges with the corresponding labels, i.e., $e_{1}^{*},\ldots,e_{k-1}^{*},e_{k+1}^{*},\ldots,e_{d}^{*}$, and their other endpoints $v_{1},\ldots,v_{k-1},v_{k+1},\ldots,v_{d}$, respectively (see Figure \ref{f:InductiveStep}).
 By construction $\beta_{e}$ crosses only the edge $e^{*}$ of $T^{*}$ and so both endpoints of the remaining edges incident to $u$ in $T^{*}$ are enclosed by $\beta_{e}$. Let $T'$ be the enclosed edges of $T^{*}$ and their endpoints, and observe that $T'$ is a tree. Thus $S=T'\setminus \{u\}$ is a forest in $T^{*}$. Denote the components of $S$ by  $S^{1},S^{2},\ldots,S^{k-1},S^{k+1},\ldots S^{d-1}$ such that $v_{i}\in S^{i}$, and let $\overline{S_{i}}=S_{i}\cup \{e_{i}^{*},u\}$ for each of $i=1,\ldots,d-1$, $i \ne k$ (an example is shown in Figure \ref{f:SiBar}).
\begin{figure}[htbp]
\begin{center}
\includegraphics[width=2in]{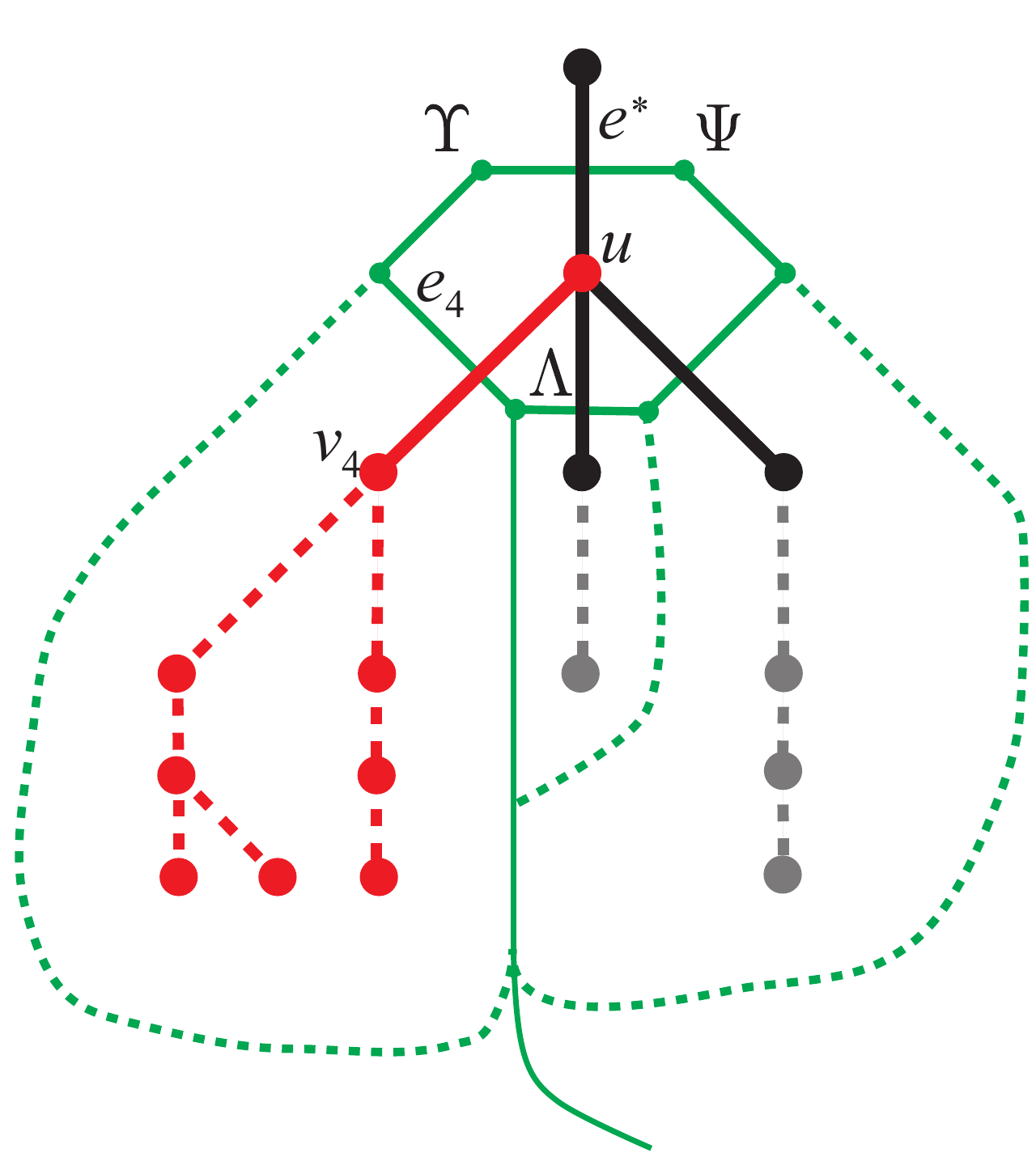}
\caption{An example shown in red of a subtree $\overline{S_{4}}$ obtained during the inductive step of Theorem \ref{l:FVOneEnough}, in this case corresponding to the vertex $v_{4}$.}
\label{f:SiBar}
\end{center}
\end{figure}

Observe that each of the trees $\overline{S_{i}}$ is the enclosed tree for the corresponding $\beta_{e_{i}}$, and so by the inductive hypothesis, each of the $w_{\beta_{e_{i}}}$ for $i=1,\ldots,d$ is equivalent to a product of the $g_{x}$ (or their inverses) suitably ordered. Also, $g_{u}\in G$ corresponds to a walk in $T$ to $\Lambda$, around  $U$, and back again. Thus $w_{\beta_{e}}$ is equal to the  product $$\displaystyle\left(\prod_{i=1}^{k-1}w_{\beta_{e_{i}}}^{s(i)}\right) g_{u}^{s(k)}\left(\prod_{i=k+1}^{d}w_{\beta_{e_{i}}}^{s(i)}\right),$$ where $s(i)$ is $\pm 1$ depending on whether the element or its inverse is required to keep the corresponding walks coherently oriented around the node $u$ (respecting the orientation induced by  $w_{\beta_{e}}$ as above). In the case of $w_{\beta_{e_{i}}}$, this would be opposite the orientation of the edge $e_{i}$, and in the case of $g_{u}$ this would be in the same orientation (see Figure \ref{f:SiOrientation}). \begin{figure}[htbp]
\begin{center}
\includegraphics[width=2in]{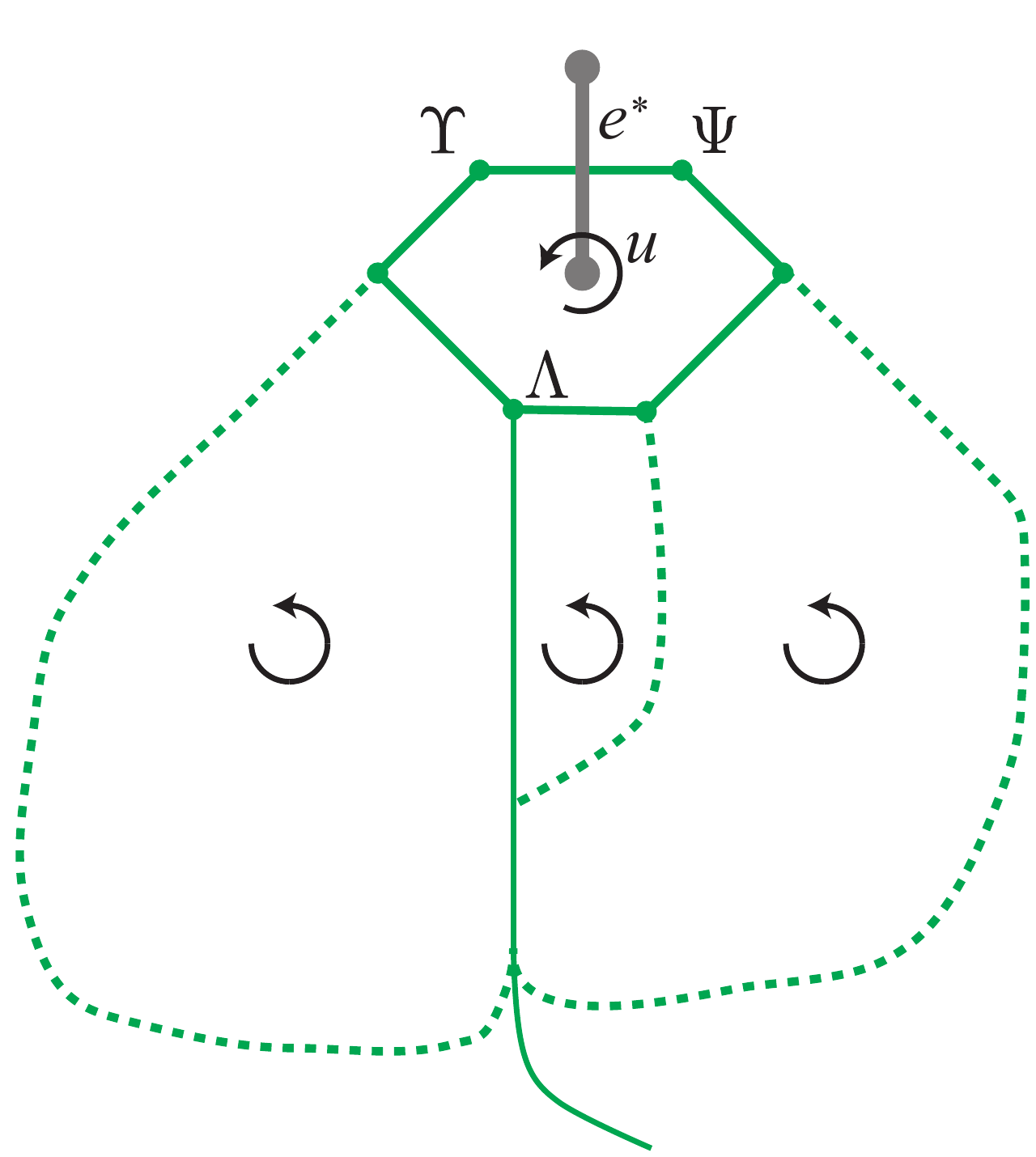}
\caption{Arrows indicate the orientation of traversal of the $w_{\beta_{e_{i}}}$ and $g_{u}$ used in the inductive step of Theorem \ref{l:FVOneEnough}.}
\label{f:SiOrientation}
\end{center}
\end{figure}
With such an orientation, group elements in the $w_{\beta_{e_{i}}}$ corresponding to shared edges in $\mc G\mc F(\mc Q)$ are cancelled by successive terms in the products and thus $w_{\beta_{e}}$ may be written as a product (suitably ordered) of elements of $G$.

Thus, by finite induction, any of the $w_{\beta_{e}}$ may be written as a product of elements in $G$. Note that any word corresponding to a walk from $\Phi$ in $T$ to a flag on a cell corresponding to an edge of $\mc Q$, around that cell and back is automatically trivial, and so may be safely omitted from the generating set for $Stab_{\Gamma}(\Phi)$. Likewise, we may omit the inverses from the description of our generating set and so we may conclude that $Stab_{\Gamma}(\Phi)=\langle g_{x}\rangle$ where $x$ is taken from the set of faces and vertices of $\mc Q$, as desired.
\end{proof}
\end{theorem}
Observe that by the argument above, a generating set for $Stab_{\Gamma}(\Phi)$ may be obtained by identifying a spanning tree $T$ in the flag graph and a set of elements of $\Gamma$ corresponding to lollipop walks with stems in $T$ rooted at $\Phi$ about the faces and vertices of $\mc Q$. This provides an algorithm for finding a small generating set for $Stab_{\Gamma}(\Phi)$. We summarize this useful fact in the following corollary.
\begin{corollary} Let $\mc Q, \mc P,\Phi, \Gamma$ as in Theorem \ref{l:FVOneEnough}. Let $T$ be a spanning tree in $\mc G\mc F(\mc Q)$. Let $G=\{g_{x}\}$ a set of elements of $\Gamma$ indexed by the set of faces and vertices in \mc Q of the form $$g_{x}=w_{x}(r_{i}r_{i+1})^{q_{x}}w_{x}^{-1}$$ with $q_{x}$ the degree of $x\in\mc G\mc F(\mc Q)^{*}$ and $w_{x}$ induces a walk in $T$ from $\Phi$ to a flag on $x$. The $Stab_{\Gamma}(\Phi)=\langle G\rangle$.\label{c:FVCorollary}
\end{corollary}
Note that in the case of  uniform polyhedra (such as prisms and antiprisms) the generating set may be reduced even further. Suppose \mc Q is a uniform polyhedron with regular cover \mc P such that  $\mc P$ has Schl\"afli type $\{p,q\}$ where $q$ is the degree of any (every) vertex in $\mc Q$. Then each of the $g_{x}$ corresponding to a vertex of $\mc Q$ in the argument above is  trivial in $\Gamma=Aut(\mc P)$ and so may be safely omitted from the generating set for $Stab_{\Gamma}(\Phi)$.

\section{Prisms}
Let $\Gamma$ be the universal regular cover of the $n$-prism ($n=3$ or $n \ge 5$), that is, the Coxeter group
\[\langle a, b, c \,|\, a^2 = b^2 = c^2 = (ac)^2 = (bc)^3 = (ab)^{\text{l.c.m.}(4,n)} = id.\rangle\]
We define type $A$ to be those flags containing a square and an edge contained in an $n$-gon. We let $g_{-1} := (ab)^{-4}$, and in general $g_k := cb(ab)^{k} c (ab)^4 c (ba)^{k} bc$ for $k = 0, 1, \dots, n-2$. We further let $h_n := c(ab)^nc$. 
\begin{proposition}\label{p:flagstabprism}
Let $g_{i},h_{n}$ as stated above, then
\[Stab_A = \langle g_k, h_n \,|\, k = -1, \dots, n-2\rangle\] for any base flag of type $A$.\end{proposition}
\begin{proof} It is immediately clear that the given elements are in the stabilizer of flags of type $A$ since they correspond to walks around either one of the bases ($h_{n}$) or to each of the square faces ($g_{k}$). Since all of the words corresponding to a walk around a vertex are trivial in $[\text{l.c.m.}(4,n),3]$, by Corollary \ref{c:FVCorollary} and the spanning tree $T$ given in Figure \ref{f:prism tree}, we know that $Stab_{A}$ is generated by $G=\{g_{k},h_{n},babc(ab)^{n}cbab\}$, since each $g\in G$ corresponds to a lollipop walk for a face of the prism with stem in $T$. It suffices therefore to demonstrate that $babc(ab)^{n}cbab$ is equivalent to some product of the $g_{k}$ and $h_{n}$.

\begin{figure}[htbp]
\begin{center}
\includegraphics[width=\textwidth]{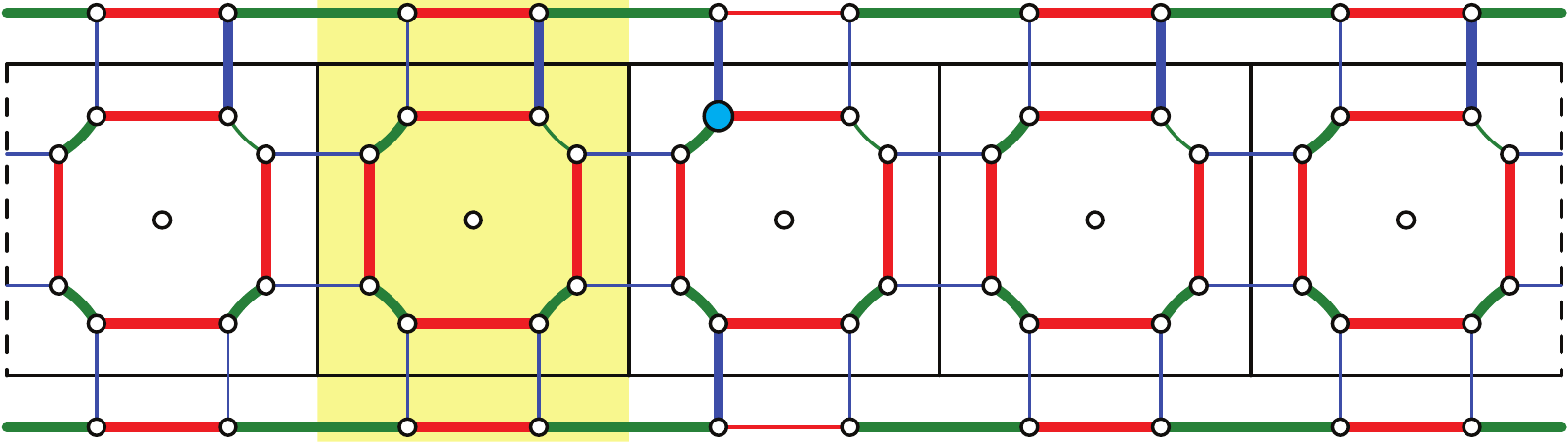}
\caption{The spanning tree of the flag graph for a 5-prism (dashed edges of are identified to construct the prism). Bold edges correspond to edges of the spanning tree $T$, while colored hairline edges correspond to the the edges in $\mc F\mc G(\mc Q)\setminus T$. The tree can be extended to arbitrary $n$-prisms by introducing (or removing) additional copies of the yellow highlighted region. Edge colors indicate the type of adjacency relationship on the flags, i.e., red corresponds to the action of $r_{0}$, green to the action of $r_{1}$ and blue to $r_{2}$.}
\label{f:prism tree}
\end{center}
\end{figure}

Consider now the product $\gamma=g_{-1}^{-1}h_{n}(\prod_{i=n-2}^{1}g_{i})g_{0}$. We first observe that
\begin{align*}
\prod_{i=n-2}^{1}g_{i}&=\prod_{i=n-2}^{1}cb(ab)^{i}c(ab)^{4}c(ba)^{i}bc=cb \left(\prod_{i=n-2}^{1}(ab)^{i}c(ab)^{4}c(ba)^{i}\right)bc\\
&=cb (ab)^{n-2}\left(\prod_{i=n-2}^{1}c(ab)^{4}cba\right)bc=cb (ab)^{n-2}\left(acb(ab)^{3}cba\right)^{n-2}bc\\
&=cb (ab)^{n-2}a\left(cb(ab)^{2}abcb\right)^{n-2}abc=c (ba)^{n-1}\left(cb(ab)^{2}acbc\right)^{n-2}abc\\
&=c (ba)^{n-1}cb\left((ab)^{2}ac\right)^{n-2}bcabc=c (ba)^{n-1}cb\left(ababca\right)^{n-2}bcabc\\
&=c (ba)^{n-1}cba\left(babc\right)^{n-2}abcabc.
\end{align*}
Thus  $\prod_{i=n-2}^{1}g_{i}$ corresponds to the walk depicted in Figure \ref{f:prismwalk1}. \begin{figure}[htbp]
\begin{center}
\includegraphics[width=\textwidth]{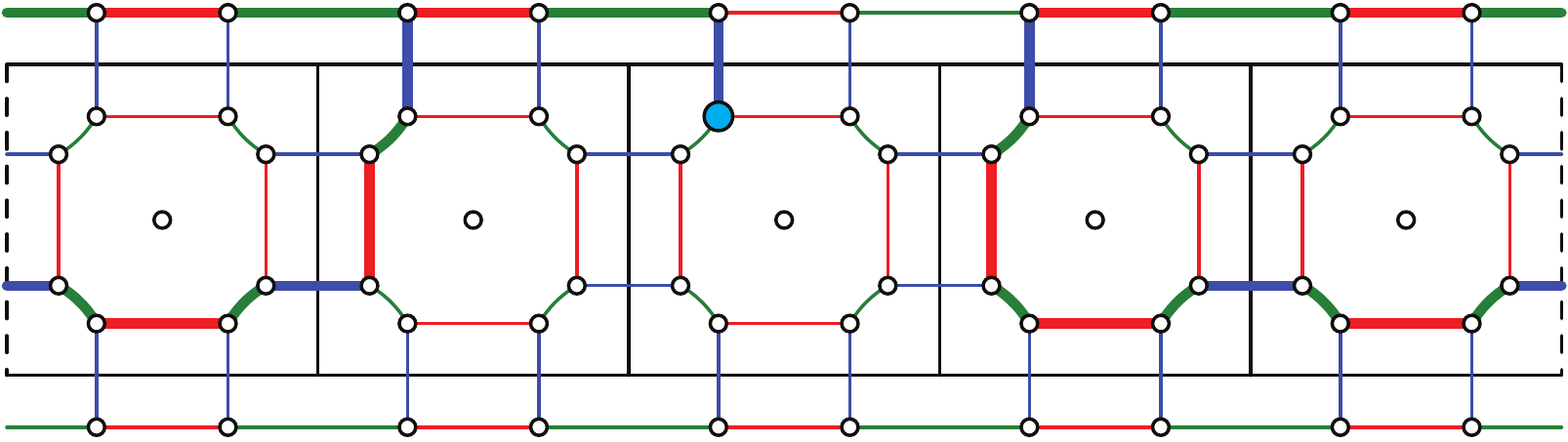}
\caption{The walk corresponding to $\prod_{i=n-2}^{1}g_{i}$ on the 5-prism, where $n=5$.}
\label{f:prismwalk1}
\end{center}
\end{figure}
We now observe that if we multiply this product on the left by $h_{n}$ and on the right by $g_{0}$ we obtain
\begin{align*} h_{n}\prod_{i=n-2}^{1}g_{i} g_{0}&=c(ab)^{n}c\cdot c (ba)^{n-1}cba\left(babc\right)^{n-2}abcabc\cdot cbc (ab)^{4}cbc\\
&=cabcba\left(babc\right)^{n-2}abcac (ab)^{4}cbc=acbcba\left(babc\right)^{n-2}aba (ab)^{4}cbc\\
&=abca\left(babc\right)^{n-2} (ba)^{2}bcbc=abac\left(babc\right)^{n-2} (ba)^{2}bcbc.
\end{align*}
This corresponds to the walk in the flag graph of the prism shown in Figure \ref{f:prismwalk2}.
\begin{figure}[htbp]
\begin{center}
\includegraphics[width=\textwidth]{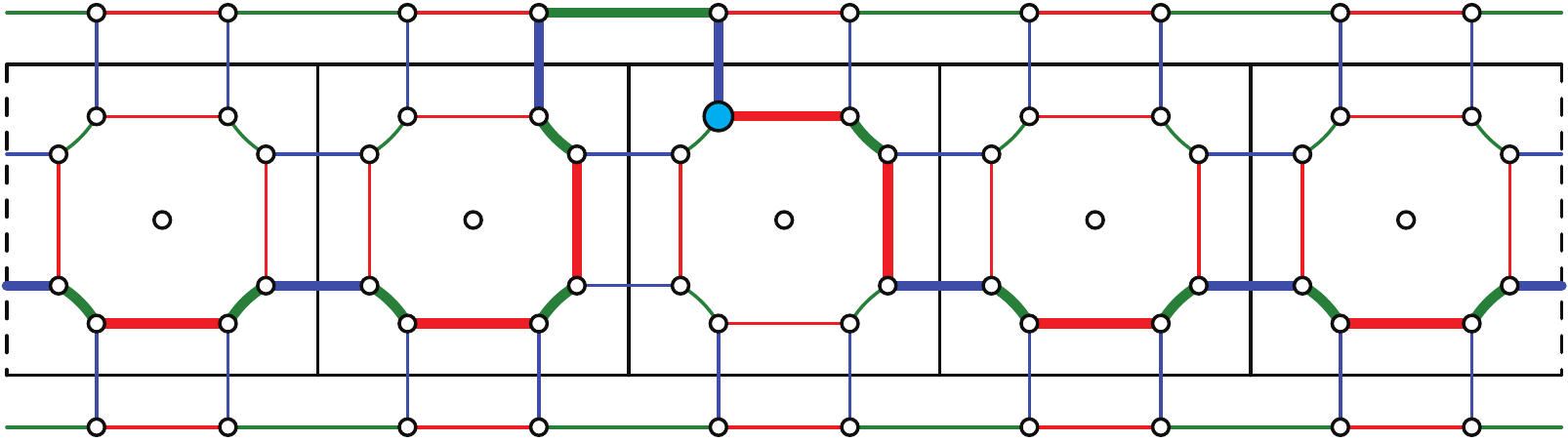}
\caption{The walk in the flag graph of the 5-prism corresponding to $abac\left(babc\right)^{n-2} (ba)^{2}bcbc$, where $n=5$.}
\label{f:prismwalk2}
\end{center}
\end{figure}
 We now multiply on the left by $g_{-1}^{-1}$, obtaining\begin{align*}
\gamma &=(ba)^{4}\cdot abac\left(babc\right)^{n-2} (ba)^{2}bcbc\\&=bababc(babc)^{n-2}babacb=ba(babc)^{n-1}babcab=ba(babc)^{n}ab\\
&=bab(abcb)^{n-1}abcab=bab(acbc)^{n-1}abcab=bab(cabc)^{n-1}abcab\\& =babc(ab)^{n-1}cabcab=babc(ab)^{n-1}acbcab=babc(ab)^{n-1}abcbab=babc(ab)^{n}cbab\end{align*} as desired.

\end{proof}

\subsection{Minimal Regular Cover of the $n$-Prism} We now proceed to prove that $\langle a,b,c|a^{2}=b^{2}=c^{2}=(ab)^{\text{\text{l.c.m.}}(4,n)}=(bc)^{3}=(ac)^{2}=w=id\rangle$, where $w=(c(ab)^{2}c(ab)^{3})^{2}$, is the automorphism group of the minimal regular cover of the $n$-prism for arbitrary values of $n\in\NN, n\ge 3$.
Let $\Gamma=\langle a,b,c|a^{2}=b^{2}=c^{2}=(ab)^{\text{\text{l.c.m.}}(4,n)}=(bc)^{3}=(ac)^{2}=id\rangle$ and denote $Cl_{\Gamma}(w)$ by $Cl^*$.

We now show that $g_k \cdot Cl^* = g_{k-2}^{-1} \cdot Cl^*$ for $i=1, \cdots, n-1$.
First we rewrite
\begin{eqnarray*}
g_k &=& cb(ab)^{k} c (ab)^4 c (ba)^{k} bc\\
&=& c(ba)^{k+1} c(ba)^{3} bcb(ab)^{k}c\\
&=& c(ba)^{k+1} c(ba)^{3} cbc(ab)^{k}c.
\end{eqnarray*}

Observe that the equation $w = id$ is equivalent to $(ba)^2c(ba)^3cb = c(ab)^3caba$.
Then,
\begin{eqnarray*}
g_k \cdot Cl^*&=& c(ba)^{k-1} ((ba)^2c(ba)^3cb)c(ab)^kc \cdot Cl^*= c(ba)^{k-1} (c(ab)^3caba)c(ab)^kc \cdot Cl^*\\
&=& c(ba)^{k-1}c(ab)^3acbcaab(ab)^{k-1}c \cdot Cl^* = c(ba)^{k-1}c(ab)^3abcbb(ab)^{k-1}c \cdot Cl^*\\
&=& c(ba)^{k-1}c(ab)^4c(ab)^{k-1}c \cdot Cl^* = c(ba)^{k-2}bacab(ab)^3c(ab)^{k-1}c \cdot Cl^*\\
&=& c(ba)^{k-2}bcb(ab)^3c(ab)^{k-1}c \cdot Cl^* = g_{k-2}^{-1} \cdot Cl^*.
\end{eqnarray*}
As a consequence,
\begin{equation}\label{eq:generatorsCl1}
Stab_A/Cl^* = \langle g_{-1} \cdot Cl^*, g_0 \cdot Cl^*, h_n \cdot Cl^*\rangle.
\end{equation}

In what follows we describe the action of the generators listed in (\ref{eq:generatorsCl1}) on the flags of type different from A. Let flags of type B be the ones containing an edge between two squares, and flags of type C those containing an $n$-gon.

Note that $g_{-1} \cdot Cl^*$ fixes flags of type B, and acts on each flag type C like a 4-step rotation of the prism. Flags of type C remain fixed under $g_0 \cdot Cl^*$ whereas each flag type B is mapped to its image under a 4-step rotation on the prism. Finally, the action of $h_n \cdot Cl^*$ on flags of type B and C depends on the congruence of $n$ $(mod \,\,4)$. If $n \equiv 0$, then $h_n$ is a trivial word in $\Gamma$; if $n \equiv 2$, then $h_n \cdot Cl^*$ acts on each flag type B and C like a half-turn with center in an adjacent square of the prism; and if $n \equiv 1, 3$ then $h_n \cdot Cl^*$ interchanges flag-orbits B and C.

We now center our attention to the $4n$-prisms. As observed in the previous paragraph, in this case $h_n$ is trivial and $Stab_A / Cl^*$ is generated by the two elements $g_{-1} \cdot Cl^*$ and $g_0 \cdot Cl^*$. Assume that an element $x = g_{-1}^{a_1} g_0^{b_1} g_{-1}^{a_2} g_0^{b_2} \cdots g_{-1}^{a_m} g_0^{b_m} \cdot Cl^* \in Stab_A / Cl^*$ acts trivially on all flags of types B and C. Because of the action described in the previous paragraph, we have that $\sum_i a_i \equiv \sum_i b_i \equiv 0$ $(mod \,\, n)$. Conversely, every word in $Stab_A / Cl^*$ satisfying the congruence relation just described acts trivially on all flags of types B and C and belongs to the core of $Stab_A$ on $\Gamma$.

\begin{proposition}
The automorphism group of the minimal regular cover of the $4n$-prism is given by the Coxeter group $[4n, 3]$ subject to the single extra relation $$(c(ab)^2c(ab)^3)^2~=~id.$$
\end{proposition}

\begin{proof}
Following the paragraph preceding the proposition, it only remains to prove that any element $x = g_{-1}^{a_1} g_0^{b_1} g_{-1}^{a_2} g_0^{b_2} \cdots g_{-1}^{a_m} g_0^{b_m} \cdot Cl^* \in Stab_A / Cl^*$ with $\sum_i a_i \equiv \sum_i b_i \equiv 0$ $(mod \,\, n)$ is trivial. It suffices to prove that $g_{-1} \cdot Cl^*$ and $g_0 \cdot Cl^*$ commute, since clearly $g_{-1}^n$ and $g_0^n$ are trivial in $\Gamma$. We frequently make use of the fact that $(bc)^3 = id$. For convenience we omit ``$\cdot Cl^*$''.
\begin{eqnarray*}
g_{-1}^{-1}g_0 &=& (ab)^4 cbc (ab)^4 cbc = (ab)^3 abbcb (ab)^3 abbcb\\
&=& (ab)^3 acb(ab)^3acb =  (ab)^3 cab(ab)^3 cab = [(ab)^3c(ab)^2](ab)^2cab.
\end{eqnarray*}

We now use the fact that $w = id$ is equivalent to $(ab)^3c(ab)^2 = c(ba)^2c(ba)^3c$.

\begin{eqnarray*}
g_{-1}^{-1}g_0 &=& (c(ba)^2c(ba)^3c)(ab)^2cab = c(ba)^2c(ba)^2bacababcab\\
&=& c(ba)^2c(ba)^2bcbabcab = cbabacbabacbcabcab\\
&=& cbabcababcabacbcab = cbabc[(ab)^2cababc]bab.
\end{eqnarray*}

Since $w = id$ is equivalent to $(ab)^2c(ab)^2c = bac(ba)^2c(ba)^3$,

\begin{eqnarray*}
g_{-1}^{-1}g_0 &=& cbabc(bac(ba)^2c(ba)^3)bab = cbabcbca(ba)^2c(ba)^4b\\
&=&cbacba(ba)^2 c(ba)^4b = cbcaba(ba)^2bbc(ba)^4b\\
&=&cbc(ab)^4 cbc (ab)^4 = g_0g_{-1}^{-1}.
\end{eqnarray*}

\end{proof}

We go now to the remaining cases of $n$-prism, that is, $n \equiv 1,2,3$ $(mod \,\,4)$.

Note that there is a rap-map from the $2n$-prism to the $n$-prism consisting in wrapping twice each $2n$-gon on itself, while identifying the squares by opposite pairs. In other words, the rap-map identifies each flag with the flag obtained from it by the half-turn $R_\pi$ whose axis contains the centers of both $2n$-gons. As a consequence, any string C-group having a flag-action on the $2n$-prism has a flag-action on the $n$-prism. In particular, the minimal regular cover of the $2n$-prism is a regular cover of the $n$-prism. We shall prove that the minimal regular covers of the $2n$-prism and of the $n$-prism coincide whenever $n$ is not a multiple of $4$.

First consider the case when $n \equiv 2$ $(mod\ 4)$, that is, $n = 2k$ for some odd integer $k$. If there is a nontrivial element $\alpha$ in the minimal regular cover of the $4k$-prism acting trivially on the $2k$-prism, then it must act on each flag of the $2k$-prism either like $id$, or like the half-turn $R_\pi$, with not all actions being trivial. We prove next that such an $\alpha$ does not exist.

Note that $\alpha$ always belongs to the stabilizer of flags of type A of the $2k$-prism, but it may or may not belong to the stabilizer $Stab_A$ of flags of type A of the $4k$-prism. We first discard the possibility of $\alpha$ belonging to $Stab_A$. Note that $g_{-1}$ fixes all flags of type B of the $4k$-prism while $g_0$ rotates them 4 steps around the prism. Similarly, $g_0$ fixes flags of type C while $g_{-1}$ rotates them 4 steps around the prism. Let $G_{Pr}$ be the subgroup  of the rotation group of the $2k$-prism generated by a 4-step rotation. Then $Stab_A$ acts on the set consisting of a given flag $\Phi_B$ of type B, a given flag $\Phi_C$ of type C, and the images of $\Phi_B$ and $\Phi_C$ by $G_{Pr}$. Since $2k$ is not a multiple of 4, $R_\pi \notin G$. As a consequence, there is no nontrivial element in $Stab_A / Cl^*$ acting on flags of type B and C either like the identity, or like $R_\pi$, but not like the identity in both kinds of flags.

The set of elements acting like $R_\pi$ on flags of type A is a right coset of $Stab_A$. In fact, we can choose that coset to be $c(ab)^kc \cdot Stab_A$. If $\alpha$ is not in $Stab_A$, then it must belong to $c(ab)^kc \cdot Stab_A$; however, all elements in $Stab_A$ preserve the $n$-gons of the $n$-prism, while $c(ab)^kc$ interchanges them. In doing so, it does not act on flags of type C like $id$ or $R_\pi$. This proves the non-existence of $\alpha$, and hence, the minimal regular cover of the $n$-prism must coincide with that of the $2n$-prism when $n = 2k$ for odd $k$.

Finally consider the case when $n$ is odd. Assume that there is an element $\alpha$ in the minimal regular cover of the $2n$-prism acting on each flag either like $id$ or like $R_\pi$.
We note that $g_{-1}$ and $g_0$ either preserve or rotate 4 steps flags of types B and C. Since $2n \equiv 2$ $(mod \,\, 4)$, then it is possible to map by a word on $g_{-1}$ and $g_0$ any flag type B (or C) into its image by all rotations by an even number of steps, but not by an odd number of steps, around the $2n$-prism. On the other hand, $h_{2n}$ is an involution which maps flags on a $2n$-gon into flags of the other $2n$-gon. These generators also satisfy the property that $h_{2n} g_i h_{2n} = g_i^{-1}$, $i=-1,0$. Since $n$ is odd, there is no element in $Stab_A$ mapping a flag type B or C of the $2n$-prism to its image by $R_\pi$, and $\alpha$ cannot belong to $Stab_A$.

Again we choose the set of elements acting like $R_\pi$ on flags of type A to be $c(ab)^nc \cdot Stab_A$, and we assume $\alpha \in c(ab)^nc \cdot Stab_A$. We note that all elements in $Stab_A$ preserve the flag-orbits of the $2n$-prism, while $c(ab)^nc$ interchanges flag orbits B and C. In doing so, it does not act on flags of type C like $id$ or $R_\pi$. This proves the non-existence of $\alpha$, and hence, the minimal regular cover of the $n$-prism must coincide with that of the $2n$-prism (and to that of the $4n$-prism) when $n$ is odd.

Overall we proved

\begin{theorem}\label{t:mincovprisms}
The automorphism group of the minimal regular cover of the $n$-prism is given by the Coxeter group $[l.c.m(4,n), 3]$ subject to the single extra relation
$(c(ab)^2c(ab)^3)^2 = id$.
\end{theorem}

At the beginning of the section we discarded the case $n=4$ (the cube) to avoid unnecessary bifurcations of the analysis. Nevertheless, Proposition \ref{p:flagstabprism} and Theorem \ref{t:mincovprisms} hold for $n=4$ as well. The arguments of the proof of Proposition \ref{p:flagstabprism} hold once it is clarified that all faces are equivalent under the automorphism group and a proper definition of the group elements $g_k$ and $h_4$ is provided. The validity of Proposition \ref{p:flagstabprism} can also be derived directly from Theorem \ref{t:treeGenerates} for a suitable spanning tree. Furthermore, Theorem \ref{t:mincovprisms} holds for $n=4$ since relation $(c(ab)^2c(ab)^3)^2 = id$ is trivial in the cube. 

To conclude this section we point out that the arguments here developed also show that the minimal regular cover of the $\infty$-prism, or the map on the plane consisting only of an infinite strip divided into infinitely many squares, is the Coxeter group $[\infty, 3]$ subject to the single relation $(c(ab)^2c(ab)^3)^2$.

\section{Antiprisms}

Let $\Gamma$ be the universal regular cover of the $n$-antiprism ($n \ge 4$), that is, the Coxeter group
\[\langle a, b, c \,|\, a^2 = b^2 = c^2 = (ac)^2 = (bc)^4 = (ab)^{l.c.m(3,n)} = id.\rangle\]
We recall that flags of type A contain a triangle and an edge contained in an $n$-gon. We let $g_{-1} := (ab)^{-3}$, $h_{-1} := bc(ab)^3cb$, and in general $g_k := cb(ab)^{k} c (ab)^3 c (ba)^{k} bc$ 
and $h_k := cb(ab)^{k} cabc (ab)^3 cbac (ba)^{k} bc$ for $k = 0, 1, \dots, n-2$. We further let $h_n := c(ab)^nc$.

\begin{proposition}\label{p:flagstabanti}
Let $g_{i},h_{i},h_{n}$ as defined above, then
\[Stab_A = \langle g_k,h_{i}, h_n \,|\, k = -1, \dots, n-2\rangle.\]\end{proposition}
\begin{proof} While it is immediately clear that the given elements are in the stabilizer of flags of type $A$ since they correspond to walks around either one of the bases ($h_{n}$) or to each of the triangular faces ($g_{k},h_{k}$), and since all of the words corresponding to a walk around a vertex are trivial in $[\text{l.c.m.}(3,n),4]$, by Corollary \ref{c:FVCorollary} and the spanning tree given in Figure \ref{f:antiprism tree} we know that $Stab_{A}$ is generated by $\{g_{k},h_{k},h_{n},bcbabc(ab)^{n}cbabcb\}$ for $-1\le l\le n-2$. It suffices therefore to demonstrate that $bcbabc(ab)^{n}cbabcb$ is equivalent to some product of the elements $g_{k}, h_{i}$ and $h_{n}$.

\begin{figure}[htbp]
\begin{center}
\includegraphics[width=\textwidth]{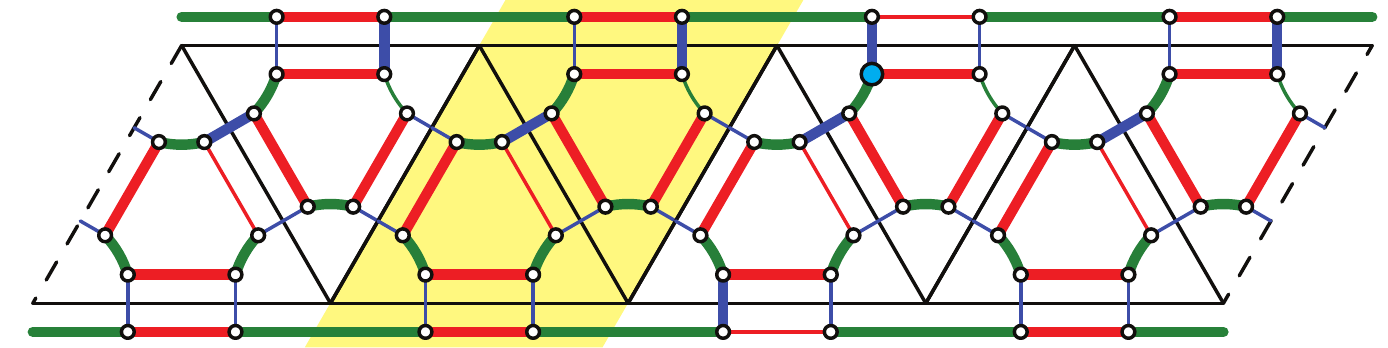}
\caption{The spanning tree of the flag graph of the 4-antiprism (dashed edges are identified to construct the anti prism). Bold edges correspond to edges of the spanning tree $T$, while colored hairline edges correspond to the the edges of $\mc F\mc G(\mc Q)\setminus T$. The tree can be extended to arbitrary $n$-antiprisms by introducing (or removing) additional copies of the yellow highlighted region.}
\label{f:antiprism tree}
\end{center}
\end{figure}
We begin by considering the product
\begin{align*}
\prod_{k=n-2}^{1}h_{k}^{-1}g_{k}&=\prod_{k=n-2}^{1}cb(ab)^{k}cabc(ba)^{3}cbac(ba)^{k}bc\cdot cb(ab)^{k}c(ab)^{3}c(ba)^{k}bc\\
&=\prod_{k=n-2}^{1}cb(ab)^{k}cabc(ba)^{3}c(ab)^{2}c(ba)^{k}bc\\&=cb\left( \prod_{k=n-2}^{1}(ab)^{k}cabc(ba)^{3}c(ab)^{2}c(ba)^{k}\right) bc\\
&=cb(ab)^{n-2}\left(\prod_{k=n-2}^{1}cabc(ba)^{3}c(ab)^{2}cba\right)bc\\
&=cb(ab)^{n-2}\left(cabc(ba)^{2}bcbabcba\right)^{n-2}bc=cb(ab)^{n-2}a\left(cbc(ba)^{2}bcbabcb\right)^{n-2}abc\\
&=c(ba)^{n-1}cbc((ba)^{2}bcbacb)^{n-3}(ba)^{2}bcbabcbabc\\
&=c(ba)^{n-1}cbcb(ababcbac)^{n-3}b(ba)^{2}bcbabcbabc\\
&=c(ba)^{n-1}cbcb(ababcbac)^{n-3}b(ba)^{2}bcbabcbabc\\
&=c(ba)^{n-1}cbcba(babcbc)^{n-3}babcbabcbabc.\\
\end{align*}
We also note that $h_{0}^{-1}g_{0}=cbcabc(ba)^{3}cbacbc\cdot cbc(ab)^{3}cbc=cbcabc(ba)^{2}bcbabcbc$, so
\begin{align*}
\gamma=h_{n}\prod_{k=n-2}^{0}h_{k}^{-1}g_{k}&=c(ab)^{n}c\cdot c(ba)^{n-1}cbcba(bacb)^{n-3}babcbabcbabc\cdot cbcabc(ba)^{2}bcbabcbc\\
&=cabcbcba(babcbc)^{n-3}babcbabcbcbc(ba)^{2}bcbabcbc\\
&=cabcbcba(babcbc)^{n-3}babcbacb(ba)^{2}bcbabcbc\\
&=cabcbcba(babcbc)^{n-3}babcbcbabcbabcbc\\
&=cabcbcba(babcbc)^{n-2}babcbabcbc.
\end{align*}
Multiplying $\gamma$ on the left by $h_{-1}^{-1}g_{-1}$ we obtain
\begin{align*}
h_{-1}^{-1}g_{-1}\gamma&=bc(ba)^{3}cb\cdot(ba)^{3}\cdot cabcbcba(babcbc)^{n-2}babcbabcbc\\
&=bc(ba)^{2}bcbabacabcbcba(babcbc)^{n-2}babcbabcbc\\
&=bc(ba)^{2}bcbaca(babcbc)^{n-2}babcbabcbc\\
&=bc(ba)^{2}bcbc(babcbc)^{n-2}babcbabcbc\\
&=bcba(babcbc)^{n-1}babcbabcbc=bcba(babcbc)^{n-1}bacbcbcabcbc\\
&=bcba(babcbc)^{n-1}bcabcbacbcbc=bcba(babcbc)^{n-1}bcabcbabcb\\
&=bcba(bacbcb)^{n-1}bcabcbabcb=bcbab(acbc)^{n-1}bbcabcbabcb\\
&=bcbab(cabc)^{n-1}cabcbabcb=bcbabc(ab)^{n-1}abcbabcb\\
&=bcbabc(ab)^{n}cbabcb
\end{align*}as desired.
\end{proof}
Then
\[Stab_A = \langle g_k, h_k, h_n \,|\, k = -1, \dots, n-2\rangle.\]

\subsection{Minimal Regular Cover of the $n$-Antiprism}
 We now proceed to prove that $\langle a,b,c|a^{2}=b^{2}=c^{2}=(ab)^{\text{\text{l.c.m.}}(3,n)}=(bc)^{4}=(ac)^{2}=w=id\rangle$, where $w=(c(ab)^2cbc(ab)^2)^2$ is the automorphism group of the minimal regular cover of the $n$-antiprism for arbitrary values of $n\in\NN, n\ge 4$.
Let $\Gamma=\langle a,b,c|a^{2}=b^{2}=c^{2}=(ab)^{\text{\text{l.c.m.}}(3,n)}=(bc)^{4}=(ac)^{2}=id\rangle$.
 Denote the word $(c(ab)^2cbc(ab)^2)^2$ by $w$, and $Cl_{\Gamma}(w)$ by $Cl_*$.

We now show that $g_k \cdot Cl_* = h_{k-2} \cdot Cl_* = g_{k-3} \cdot Cl_*$ for all $k=2, \dots, n-1$. For convenience we omit ``$\cdot Cl_*$'' at the end of each group element. We frequently use the following fact.
\[id = w = caba(bcbc)abab(ca)babcbcabab = caba(cbcb)abab(ac)babcbcabab.\]
Taking the inverse we get
\[id = babacbcbabc(ab)^3cbcabac,\]
and therefore
\begin{equation}\label{eq:antiprism}
(ab)^3 = cbabcbcababcabacbc.\end{equation}

\begin{eqnarray*}
g_k &=& cb(ab)^kc(ab)^3c(ba)^kbc = cb(ab)^kc(cbabcbcababcabacbc)c(ba)^kbc\\
&=& cb(ab)^kbabcbcababcabcab(ba)^kbc = cb(ab)^{k-1}bcbcababacbc(ba)^{k-1}bc\\
&=& cb(ab)^{k-2}acbcabababbcbc(ba)^{k-1}bc = cb(ab)^{k-2}acbc(ab)^3bcbc(ba)^{k-1}bc\\
&=& cb(ab)^{k-2}cabc(ab)^3cbcb(ba)^{k-1}bc = cb(ab)^{k-2}cabc(ab)^3cbca(ba)^{k-2}bc = h_{k-2}\\
&=& cb (ab)^{k-2}cabc(cbabcbcababcabacbc)cbac(ba)^{k-2}bc = cb (ab)^{k-2}(cbcbc)ababcab(ba)^{k-2}bc\\
&=& cb(ab)^{k-2}(bcb)ababc(ba)^{k-3}bc = cb(ab)^{k-3}acbababc(ba)^{k-3}bc\\
&=& cb(ab)^{k-3}c(ab)^3c(ba)^{k-3}bc = g_{k-3}.
\end{eqnarray*}

Hence $Stab_A/Cl_*$ is generated only by the four elements $g_{-1} \cdot Cl_*, g_0 \cdot Cl_*, h_{-1}\cdot Cl_*, h_n\cdot Cl_*$.

We denote by flags of type B those $1$-adjacent to flags of type A, flags of type C those $2$-adjacent to flags of type B, and flags of type D those $2$-adjacent to flags of type A. Then $g_{-1} \cdot Cl_*$ acts like $id$ on flags of type B and C, $g_0 \cdot Cl_*$ acts like $id$ on flags of types B and D, and $h_{-1} \cdot Cl_*$ acts like $id$ on flags of types C and $D$. Each of $g_{-1} \cdot Cl_*$, $g_0 \cdot Cl_*$ and $h_{-1} \cdot Cl_*$ act like a three step rotation around the antiprism on flags of the (unique) type they do not fix. On the other hand, $h_n$ is trivial if $n \equiv 0$ $(mod \,\,3)$; otherwise $h_n \cdot Cl_*$ does not preserve flag orbits B, C and D.

We first analyze the case of the $3n$-prism. According to the action of the three generators of $Stab_A/Cl_*$ on flags of types B, C and D, every element $\alpha$ fixing all flags must be such that the sum of the exponents of all factors $g_{-1} \cdot Cl_*$ (resp. $g_0 \cdot Cl_*$, $h_{-1} \cdot Cl_*$) on any word corresponding to $\alpha$ must be a multiple of $n$. Conversely, any element $\alpha \in Stab_A/Cl_*$ such that all words representing $\alpha$ satisfy the property just described must preserve all flags. We shall prove that relation $w = id$ implies that all such elements $\alpha$ are trivial in $Stab_A/Cl_*$, implying in turn that $Cl_* = Core_\Gamma(Stab_A)$. To do this, it suffices to verify that the elements $g_{-1}$, $g_0$ and $h_{-1}$ commute, or equivalently, that the elements $g_{-1}^{-1} h_{-1} g_{-1} h_{-1}^{-1}\cdot Cl_*$, $g_{-1}^{-1} g_0 g_{-1} g_0^{-1}\cdot Cl_*$ and $h_{-1} g_0 h_{-1}^{-1} g_{0}^{-1} \cdot Cl_*$ are trivial.

Note that $(bc)^4 = id$, and that relation $w = id$ is equivalent to $ababcabab = cbcbabacbabacbc$ and to $cbcababcab = babacbabacbcba$. Assuming this, and omitting ``$\cdot Cl_*$'' for convenience, we have

\begin{eqnarray*}
g_{-1}^{-1} h_{-1} g_{-1} h_{-1}^{-1} &=& (ab)^3bc(ab)^3cb(ba)^3bc(ba)^3cb = ababacabababcabababcbababacb\\
&=& ababcb(ababcabab)abcbababacb \\&=& ababcb(cbcbabacbabacbc)abcbababacb\\
&=& abacabacbabacbcabcbababacb = abcbacbabacbcabcbababacb\\
&=& abcbcababacbacbcbababcab = abcbcababacbab(cbcababcab)\\
&=& abcbcababacbab(babacbabacbcba) = abcbcababacacbabacbcba\\
&=& abcbcababbabacbcba = id.
\end{eqnarray*}

Similarly, since $w \cdot Cl_* = id$ is equivalent to $ababcbcababcab \cdot Cl_*= cbabacbcba \cdot Cl_*$, by omitting ``$\cdot Cl_*$'' we have

\begin{eqnarray*}
g_{-1}^{-1} g_0 g_{-1} g_0^{-1} &=& (ab)^3 cbc (ab)^3 cbc (ba)^3 cbc (ba)^3 cbc \\&=& abababcbcababacbcababacbcbababacbc\\
&=& abababcbcababacbcab(ababcbcababcab)c\\& = &abababcbcababacbacb(cbabacbcba)c\\
&=& abababcbcababacbacbcbabacbcbac \\&=& abababcbcababcababcbcababcbcac\\
&=& ab(ababcbcababcababcbcababc)ba = ab(w)ba = id.
\end{eqnarray*}

Finally, since $w \cdot Cl_* = id$ is equivalent to $ababcababc \cdot Cl_* = cbcbabacbabacb \cdot Cl_*$ and $cababcbcababcababc\cdot Cl_{*}=babacb\cdot Cl_{*}$, and by omitting ``$\cdot Cl_*$'' we obtain

\begin{eqnarray*}
h_{-1} g_0 h_{-1}^{-1} g_{-1}^{-1} &=& bc(ab)^3cbcbc(ab)^3cbcbc(ba)^3cbcbc(ba)^3cbc\\
&=& bcabababbcbabababbcbbabababcbbababacbc\\& =& bcababacbababacabababcababacbc\\
&=& bcababcabababcb(ababcababc)abc\\&=& bcababcabababcb(cbcbabacbabacb)abc\\
&=& bcababcababacabcababacbabc = bcabab(cababcbcababcababc)\\
&=& bcabab(babacb) = id.
\end{eqnarray*}

We have proved

\begin{proposition}
The automorphism group of the minimal regular cover of the $3n$-antiprism is given by the Coxeter group $[3n, 4]$ subject to the single extra relation
$(c(ab)^2cbc(ab)^2)^2 = id$.
\end{proposition}

Whenever $m \equiv 1,2$ $(mod \,\,3)$ we let $n = 3m$ and note that there is a rap-map from the $n$-prism to the $m$-prism, where the preimage of any flag in the $m$-prism is a set of three flags with the property that they can be obtained from each other by an $m$-step rotation or a $2m$-step rotation around the $3m$-antiprism. Consequently, the minimal regular cover of the $n$-antiprism covers the minimal regular cover of the $m$-prism, and the kernel of this cover is the set of elements in the minimal regular cover of the $n$-antiprism with trivial action on all flags of the $m$-antiprism, that is, all elements in the minimal regular cover of the $n$-antiprism which map each flag of the $n$-antiprism to itself, or to its image under the $m$-steps or $2m$-steps rotations. As we shall see, $id$ is the only element satisfying this property.

Assume that $\alpha$ is an element in automorphism group of the minimal regular cover of the $n$-antiprism mapping each flag of the $n$-antiprism to itself, or to its image under the $m$-steps or $2m$-steps rotations. If $\alpha$ belongs to the stabilizer $Stab_A$ of a flag type $A$ of the $n$-prism 
then, the action of the generators of $Stab_A$ on the flags of the $n$-prism described above implies that $\alpha$ must map each flag of type B, C or D to its image under a $3k$-step rotation for some integer $k$. Since $m$ is not a multiple of $3$, then $\alpha$ must act like $id$. On the other hand, if $\alpha \notin Stab_A$, then $\alpha$ must belong either to $c(ab)^mc \cdot Stab_A$ or to $c(ab)^{2m}c \cdot Stab_A$. However, no element in those two cosets preserves the flag orbits $B, C$ and $D$, contradicting our assumption about the action of $\alpha$ on the flags of the $3n$-prism.

Overall we proved

\begin{theorem}\label{t:mincovanti}
The automorphism group of the minimal regular cover of the $n$-antiprism is given by the Coxeter group $[l.c.m(3,n), 4]$ subject to the single extra relation
$(c(ab)^2cbc(ab)^2)^2 = id$.
\end{theorem}

The 3-antiprism is isomorphic to the octahedron, which is a regular polyhedron. Similarly to the case of the prisms, Proposition \ref{p:flagstabanti} and Theorem \ref{t:mincovanti} also hold in this case.

To conclude we note that the arguments developed above show that the minimal regular cover of the ``$\infty$-antiprism'' (see Figure \ref{fig:antiprism}), is the Coxeter group $[\infty, 4]$ subject to the unique relation $(c(ab)^2cbc(ab)^2)^2$.

\begin{figure}
\begin{center}
\includegraphics[width=14cm, height=1.5cm]{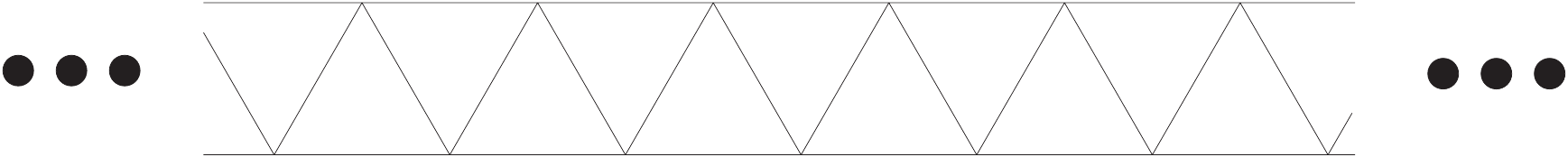}
\caption{Infinite antiprism\label{fig:antiprism}}
\end{center}
\end{figure}
\section{Topology and Algebraic Structure of the Minimal Covers}
We start our discussion by observing that the automorphism group of a minimal regular cover of an $n$-prism is determined by the $\text{l.c.m.}(4,n)$, so for  $m$ odd and $n=m, 2m, 4m$, the corresponding $n$-prisms share the same minimal regular cover. Thus we need only concern ourselves with studying the structure of the covers of $4m$-prisms for arbitrary $m\in\NN$. The following theorem describes the group structure of the monodromy group of the $4m$-prism.
\begin{theorem}\label{t:groupprisms}
Let $\mc P_{4m}$ be the minimal regular cover of the $4m$ prism. Then $\Gamma(\mc P_{4m})$ contains a normal subgroup $H$ isomorphic to $\ZZ_{m}^{3}$, and the quotient $\Gamma({\mc P}_{4m})/H$ is isomorphic to the octahedral group $B_{3}$. In particular, $\Gamma(\mc P_{4m})$ has order $48m^{3}$.
\end{theorem}

\begin{proof}
We abuse notation and denote by $a, b, c$ the generators of $\Gamma(\mc P_{4m})$, which coincides with the monodromy group of the $4m$-prism.
Let $\alpha := (ab)^4, \beta := c(ab)^4c, \gamma := bc(ab)^4cb \in \Gamma(\mc P_{4m})$. We claim that $H := \langle \alpha, \beta, \gamma \rangle \cong \ZZ_{m}^3$. To see this note first that the order of $\alpha$, $\beta$ and $\gamma$ is $m$. Recall that an element of $\Gamma(\mc P_{4m})$ fixes all flags of the $4m$-prism if and only if it is the identity element, and note that the action of the commutators $\alpha \beta \alpha^{-1} \beta^{-1}$, $\alpha \gamma \alpha^{-1} \gamma^{-1}$ and $\beta \gamma \beta^{-1} \gamma^{-1}$ on all flags of the $4m$-prism is trivial. Moreover, the elements in $\langle \alpha \rangle$ fix flags of types A and B, the elements in $\langle \beta \rangle$ fix flags of types B and C, and the elements in $\langle \gamma \rangle$ fix flags of types A and C, so these three subgroups have trivial intersection. This implies the desired isomorphism.

To see that $H$ is normal in $\Gamma(\mc P_{4m})$ it suffices to see that the conjugates of $\alpha$, $\beta$ and $\gamma$ by $a$, $b$ and $c$ belong to $H$. In all cases the computations are straightforward except, perhaps, for $a\gamma a$. Note that $w=id$ implies that $abc(ab)^3 = bac(ba)^3c(ba)^2c$, and then
\begin{eqnarray*}
a\gamma a &=& abc(ab)^3\cdot abcba = bac(ba)^3c(ba)^2c \cdot abcba = bca(ba)^3cbabacabcba\\ &=& bca(ba)^3cbabcbcba
  = bc(ab)^3cbaca = bc(ab)^3cb = \gamma.
\end{eqnarray*}
Alternatively one can note that the orbit of each flag under the action of $H$ coincides with its orbit under $aHa$, $bHb$ and $cHc$.

Note that all elements in $H$ stabilize all flag orbits. On the other hand, $a$ fixes all flag orbits, $b$ interchanges flags of type A with flags of type B, and $c$ interchanges flags of type A with flags of type C. As a consequence, $a \cdot H$, $b \cdot H$ and $c \cdot H$ are three different elements in $\Gamma(\mc P_{4m})/H$. Furthermore, $\{a \cdot H, b \cdot H, c \cdot H\}$ is a generating set of $\Gamma(\mc P_{4m})/H$ consisting of three involutions, two of which commute. The order of $bc \cdot H$ must divide $3$, and the order of $ab \cdot H$ must divide $4$, since $(ab)^4 \in H$. By observing the action of $(ab)^k \cdot H$ and of $(bc)^k \cdot H$ for $k = 1, 2, 3$ we note that the order  of these elements is $4$ and $3$ respectively. Then $\Gamma(\mc P_{4m})/H$ must be a subgroup of the symmetric group of the cube containing an element of order $3$ and an element of order $4$. It follows that $\Gamma(\mc P_{4m})/H$ is either the symmetry group of the cube, or the symmetry group of the hemicube. However, the order of $abc \cdot H$ is $6$ and not $3$, discarding the latter. This finishes the proof, since the octahedral group is isomorphic to the symmetry group of the cube.
\end{proof}

The symmetry group of the toroidal $4$-polytope $\{4,3,4\}_{(m,0,0)}$ in the notation of \cite[Section 6]{McMSch02} 
is $\ZZ_m^{3} \rtimes B^3$. We note that this group is not ismorphic to $\Gamma(\mc P_{4m})$ since $\Gamma(\{4,3,4\}_{(m,0,0)})$ contains no central element, and $(abc)^{3m} \in Z(\Gamma(\mc P_{4m}))$. In fact, $(abc)^{3m}$ acts on all flags of the $4m$-prism as the half-turn with respect to the axis through the centers of the two $4m$-gons, and hence it commutes with all elements of the monodromy group of the $4m$-prism.

Now that we know the structure of the group, we can discuss the topological structure of $\mc P_{4m}$. We first note that since $w$ (and any other element of $Cl_{\Gamma}(w)$) is a product of an even number of generators, the corresponding quotient is orientation preserving and so $\mc P_{4m}$ lies on an orientable surface. The polyhedron is regular, the number of flags in $\mc P_{4m}$ is $48m^{3}$, so the surface where $\mc P_{4m}$ lies is compact ($\mc P_{4m}$ has a finite number of flags). We also observe that
\bi
\item there are $2\cdot 4m$ flags per face,
\item there are 4 flags per edge,
\item and there are 6 flags per vertex.
\ei Thus the number of faces is $6m^{2}$, the number of edges is $12m^{3}$, and the number of vertices is $8m^{3}$, and so the Euler characteristic of the surface is given by $\chi(\mc P_{4m})=6m^{2}-12m^{3}+8m^{3}=(6-4m)m^{2}$. Thus $\mc P_{4m}$ lies on a compact orientable surface of genus $(2m-3)m^{2}+1$.

We may engage in a similar line of reasoning as regards the $n$-antiprism. Here the minimal regular cover is determined by the $\text{l.c.m.}(3,n)$, so for $n=m$ and $n= 3m$, the corresponding $n$-antiprisms share the same minimal regular cover whenever  $m\not\equiv 0 \ (\bmod 3)$. Thus we need only be concerned with minimal regular covers for the $3m$-antiprims, with $m\in\NN$.
The following theorem describes the group structure of the monodromy group of the $3m$-antiprism.
\begin{theorem}
Let $\mc A_{3m}$ be the minimal regular cover of the $3m$-antiprism. Then $\Gamma(\mc A_{3m})$ contains a normal subgroup $K$ isomorphic to $\ZZ_{m}^{4}$. Furthermore, the quotient $\Gamma(\mc A_{3m})/K$ is isomorphic to the octahedral group $B_3$. In particular, $\Gamma(\mc A_{3m})$ has order $48m^{4}$.
\end{theorem}

\begin{proof}
   The proof follows from similar arguments to those of the proof of Theorem \ref{t:groupprisms}.

   Consider now $K := \langle (ab)^3, c(ab)^3c, bc(ab)^3cb, cbc(ab)^3cbc \rangle$ and use similar considerations to those in the proof of Theorem \ref{t:groupprisms} to show that $K \cong \ZZ_{m}^4$.

   To verify that $K$ is normal it can be done by noting that $aKa$, $bKb$ and $cKc$ induce the same orbit as $K$ on any given flag of the $3m$-prism. Alternatively it can be done algebraically, where, using (\ref{eq:antiprism}),
\begin{eqnarray*}
a \cdot bc(ab)^3cb \cdot a &=& abc \cdot (ab)^\cdot cba = abc \cdot cbabcbcababcabacbc \cdot cba\\
 &=& bcbcababcabaca = cbcbababacbc = (cbc (ab)^3 cbc)^{-1},
\end{eqnarray*}
and hence $a \cdot cbc(ab)^3cbc \cdot a = (bc(ab)^3cb)^{-1}$.

   Finally, the quotient $\Gamma(\mc A_{3m})/K$ is isomorphic to $B_3$ since $(ab)^3 \cdot K$ and $(bc)^4 \cdot K$ have order 3 and 4 respectively, whereas $abc \cdot K$ has order 6.
\end{proof}

As in the case of the prisms above, the $w$ in the automorphism group of the minimal regular cover of the $n$-antiprism is a product of an even number of generators and so the corresponding quotient from the covering hyperbolic tiling is orientation preserving, so $\mc A_{3m}$ lies on a compact orientable surface. We also observe that there are $6 m$ flags per face, $4$ per edge, and $8$ per vertex. Thus the number of faces is $8m^{3}$, the number of edges is $12 m^{4}$ and the number of vertices is $6m^{4}$. Thus the Euler characteristic of the surface is given by $\chi(\mc A_{3m})=8 m^{3}-12 m^{4}+6 m^{4}$, and so the genus of $\mc A_{3m}$ is $3m^{4}-4m^{3}+1$. 
\section{Discussion of Results}
In \cite{HarWil10, PelWil10} and ~\cite{PelWil11} generating sets for the stabilizer of a base flag of a polyhedron in the automorphism group of the regular cover were obtained by considering just one generator (at most) per face of the polyhedron. In particular, these generators correspond to lollipop walks. For the finite polyhedra in \cite{HarWil10}, confirmation that this set of generators was adequate to generate the stabilizer had to be confirmed computationally using GAP\nocite{GAP4}, while for the infinite polyhedra in \cite{PelWil10, PelWil11}  we had to rely on a carefully constructed spanning trees and the application of Theorem \ref{t:treeGenerates} to demonstrate sufficiency. It is not, however, reasonable to suppose that such a set of generators would suffice in general, even if one also includes all of the generators corresponding to lollipop walks for the vertices. In particular, counterexamples may easily be obtained via consideration of polyhedral maps on the projective plane. Thus Theorem \ref{l:FVOneEnough} provides a sufficiency condition for generating sets for polyhedra with planar flag graphs, in particular the lemma shows the sufficiency in general of a much smaller set of generators than those suggested by Theorem \ref{t:treeGenerates} for spherical and planar polyhedra, but some additional questions remain in this area requiring further investigation. For the polyhedra with planar flag graphs will any set of generators corresponding to one lollipop walk per face and vertex of the polyhedron work, or must the stems of the lollipops all belong to a tree (as required by the proof of the theorem and noted in Corollary \ref{c:FVCorollary})? For polyhedra of other topological types, what conditions are necessary for a collection of generators to guarantee that they suffice to generate the stabilizer of a base flag, in particular, are there correspondingly small sets of generators (e.g., one corresponding to each lollipop walk around a face or vertex of the polyhedron, plus some small number depending on the genus)? Likewise, little, if anything, seems to be known about sufficiency theorems for generating sets for the stabilizer of a base flag of abstract polytopes of higher rank, where upper bounds are given by the generating sets given by Theorem \ref{t:treeGenerates}. Finding such small generating sets can be instrumental in characterizing the structure of minimal regular covers because they significantly reduce the complexity of the associated computations. Thus it is an open question  whether one may  determine, based on  geometric features of a polytope (e.g., rank, number of facets and/or vertices, etc.), a small upper bound on the number of generators for the stabilizer of a base flag of a finite polytope. 

\section{Acknowledgements}
The authors would like to thank Barry Monson for his many kindnesses, and his helpful suggestions for this project. The authors would also like to thank the Fields Institute for providing valuable opportunities to discuss and expand this research during the Workshop on Symmetry in Graphs, Maps and Polytopes at the Institute.

The work of Daniel Pellicer  was supported by IACOD-M\'exico under project grant IA101311.

\bibliographystyle{ieeetr}

 \end{document}